\documentclass[12pt,reqno]{amsart} 

\usepackage{amssymb,latexsym}
\usepackage[draft=true]{hyperref}
\usepackage{cite} 

\usepackage[height=190mm,width=130mm]{geometry} 
\usepackage{graphicx}

\theoremstyle{plain}

\theoremstyle{definition}

\theoremstyle{remark}

\newtheorem{example}{Example}

\numberwithin{equation}{section} 

\newtheorem{lem}{\textbf{Lemma}}[section]
\newtheorem{thm}[lem]{\textbf{Theorem}}
\newtheorem{co}[lem]{\textbf{Corollary}}

\newtheorem{prop}[lem]{\textbf{Proposition}}

\theoremstyle{definition}

\theoremstyle{definition}
\theoremstyle{remark}

\makeatletter

\begin{document}
\title[An implicit algorithm  for  a  $Q$-nonexpansive mapping]{An implicit algorithm  for finding a   fixed point of a  $Q$-nonexpansive mapping in locally convex spaces} 

\author{Ebrahim  Soori,
M. R. Omidi,
Ali Farajzadeh, Yuanheng Wang }

\address{ {  Lorestan University, Department  of Mathematics,  P. O. Box 465, Khoramabad, Lorestan, Iran(E. Soori)\\
Department of Basic Sciences, Kermanshah University of Technology, Kermanshah, Iran(M. R. Omidi)\\
Department of Mathematics, Razi University, Kermanshah, Iran(Ali Farajzadeh)\\
Department of Mathematics, Zhejiang Normal university, Jinhua, China. This work was supported by the National Natural Science Foundation of China (Yuanheng Wang) (no. 11671365).}}

\email{  sori.e@lu.ac.ir, : m.omidi@kut.ac.ir,   A.Farajzadeh@razi.ac.ir,  yhwang@zjnu.cn }






\begin{abstract}
Suppose that  $Q$ is a   family of   seminorms on a             locally
convex space $E$ which determines the topology of $E$. In this paper, first  we define the  notation of the $q$-duality mappings in locally convex spaces.
Then we introduce     an implicit method for finding an element of
the set  of   fixed points
of   a $Q$-nonexpansive mapping. Then we prove the   convergence of the proposed implicit scheme to a   fixed point  of      the $Q$-nonexpansive mapping in $\tau_{Q}$.
\end{abstract}


\subjclass[2010]{47H10}

\keywords{ Fixed point, $Q$-Nonexpansive mappings, Duality mappings,  Seminorms,    $Q$-contraction mappings.}

\maketitle

\section{ Introduction}
Let $C$ be a nonempty closed and convex subset of a Banach space $E$ and
$E^{*}$ be the dual space of $E$. Let $\langle.,.\rangle$   denote the pairing between $E$ and $E^{*}$. The
normalized duality mapping $J: E \rightarrow E^{*}$
is defined by
\begin{align*}
    J(x)=\{f \in E^{*}: \langle x, f \rangle= \|x\|^{2}=\|f\|^{2} \}
\end{align*}
for all $x \in E$. In this investigation we study duality mappings for locally convex spaces that will be denoted by $J_{q}$ for a seminorm $q$.

Suppose that  $Q$ is a   family of   seminorms on a             locally
convex space $E$ which determines the topology of $E$ that will be denoted by  $\tau_{Q}$.
Let $C$ be a nonempty closed and convex subset of  $E$.  A mapping $ T$ of $ C $ into itself is called $Q$-nonexpansive if $q(Tx - Ty) \leq q(x - y)$, for all $x, y \in C$ and $q \in Q$, and a mapping $f$ is a $Q$-contraction on $E $ if  $ q(f (x) -f (y)) \leq \beta q(x - y)$,  for all $x, y \in E$  such that  $0 \leq \beta < 1$.

In this paper   we introduce the following general implicit algorithm for finding an   element of the set of    fixed points
of   a  $Q$-nonexpansive  mapping. On the other hand, our goal is to prove that there  exists   a sunny $Q$-nonexpansive retraction $ P $  of $ C $ onto $ \rm{Fix(T)}$ and $ x \in C $  such that the following  sequence  $\{z_{n}\}$     converges     to   $ Px $ in $\tau_{Q}$.
 \begin{equation*}
    z_{n}=\epsilon_{n} fz_{n}+(1-\epsilon_{n})Tz_{n}\quad ( n \in  I),
    \end{equation*}
where $f$ is a $Q$-contraction and $T$ is a $Q$-nonexpansive mapping. To receive to
  the  aim, some new concepts in locally convex spaces  will be devised.   For example,   some new results of   Hahn Banach theorem and   Banach contraction principle will be   generalized to    locally
convex spaces.

\section{Methods and preliminaries}

Suppose that  $Q$ is a   family of   seminorms on a             locally
convex space $E$ which determines the topology of $E$ that will be denoted by  $\tau_{Q}$.

Let $D$ be a subset of $B$ where $B$ is a subset of a  locally
convex space $E$ and let $P$ be a retraction of $B$ onto $D$, that is,
$Px = x$ for each $x \in D$. Then $P$ is said to be sunny, if for each $x \in B$ and $t\geq0$ with $Px + t (x - Px) \in B$,
$P(Px + t (x - Px)) = Px$.
A subset $D$ of $B$ is said to be a sunny  $Q$-nonexpansive retract of $B$ if there exists a sunny  $Q$-nonexpansive retraction $P$ of
$B$ onto $D$.
We know that for a Banach space  if $E$ is smooth and $P$ is a retraction of $B$ onto $D$, then $P$ is sunny and   nonexpansive if and
only if for each $x \in B$ and $z \in D$,
\begin{equation}\label{qpgfxm}
   \langle  x-Px , J(z-Px)\rangle    \leq 0.
\end{equation}
For more details, see  \cite{tn}. In this paper, after defining the notation of $q$-duality mappings,  we prove  inequality  \eqref{qpgfxm} for locally convex spaces then we will prove our implicit  algorithm.

   Recall the following definitions:
\begin{enumerate}
\item  The locally convex topology $\tau_{Q}$ is separated if and only if the family of seminorms
$Q$ possesses the following property:
for each  $x \in  E \backslash \{0\}$ there exists $q \in Q$ such that $q(x)  \neq 0$ or equivalently \\ $\displaystyle\bigcap_{q \in Q} \{ x \in E: q(x)=0\}= \{0\}$ ( see \cite{barbu}),
\item let $E$ be a locally convex topological vector space over $\mathbb{R}$
or $\mathbb{C}$. If $U \subset E$, then the polar of $U$, denoted by $U^{\circ}$, is the  set
$$\{f \in E^{*}: |f(x)| \leq 1, \forall x \in U\}.$$
\end{enumerate}

 Suppose that  $Q$ is a   family of   seminorms on a         locally convex space
 $X$ which determines the topology of $X$ and $q \in Q$ is a    seminorm. Let $Y$ be a    subset  of
 $X$, we put   $q^{*}_{Y}(f)=\sup\{|f(y)|: y \in Y, q(y)\leq 1\}$ and      $q^{*}(f)=\sup\{|f(x)|: x \in X, q(x)\leq 1\}$, for every linear functional $f$ on $X$. Observe that, for each $x \in X$ that $q(x) \neq 0$ and $f \in X^{*}$, then   $|\langle x , f \rangle | \leq q(x) q^{*}(f)$. For more details, see \cite{soori}.
 We will make use of the following Theorems.
\begin{thm}\cite{soori}\label{hahn1}
  Suppose that  $Q$ is a   family of   seminorms on a   real     locally
convex space $X$ which determines the topology of $X$ and $q \in Q$ is a continuous  seminorm and $Y$ is a  vector subspace of $X$ such that $Y \cap \{x \in X: q(x)=0\}=\{0\}$.     Let $f$ be a real    linear functional on $Y$ such that   $q^{*}_{Y}(f)< \infty$. Then there exists a continuous linear functional $h$ on $X$ that extends $f$ such that $q^{*}_{Y}(f)=q^{*}(h)$.
\end{thm}

\begin{thm}\cite{soori}\label{hahn2}
   Suppose that  $Q$ is a   family of   seminorms on a   real     locally
convex space $X$ which determines the topology of $X$ and  $q \in Q$   a nonzero  continuous seminorm.    Let $x_{0}$ be a point in $X$.  Then there exists  a continuous linear functional on $X$ such that $q^{*}(f)=1$ and  $f(x_{0})=q(x_{0})$.
\end{thm}

\section{Main result}

First we   define our  notation of the $q$-duality mapping: \\
Suppose that  $Q$ is a   family of   seminorms on a   real     locally
convex space $X$ which determines the topology of $X$, $q \in Q$ is a continuous  seminorm and   $X^{*}$ is the dual space of  $X$.     A multivalued mapping $J_{q} : X \rightarrow 2^{X^{*}}$
    defined by $$J_{q}x =\{ j \in X^{*}: \langle x , j\rangle ={q(x)}^{2}={q^{*}(j)}^{2} \},$$  is called      $q$-duality mapping. Obviously, $J_{q}(-x) = -J_{q}(x)$.      $J_{q}x \neq \emptyset$. Indeed,   let  $x \in X$, if $q(x)=0$, $j=0$ is in $J_{q}x$, and  in other words, if  $q(x)\neq 0$,   from Theorem \ref{hahn2}, there exists    a linear functional $f \in  {X^{*}}$ such that $q^{*}(f)=1$ and  $\langle x , f \rangle=q(x)$. Putting $j:=q(x)f$, we have $$\langle x , j\rangle =\langle x , q(x)f \rangle=q(x)\langle x ,  f \rangle={q(x)}^{2},$$ and we have also,
  \begin{align*}
   q^{*}(j)=&\sup\{|j(y)|: y \in X, q(y)\leq 1\}=\sup\{|q(x)f(y)|: y \in X, q(y)\leq 1\}\\=&q(x)\sup\{|f(y)|: y \in X, q(y)\leq 1\}=q(x)q^{*}(f)=q(x).
  \end{align*}

\begin{lem}\label{qupozm}
  Suppose that  $Q$ is a   family of   seminorms on a   Hausdorff and   complete   locally
convex space $E$ which determines the topology of $E$. Let  $\phi_{q} : E \rightarrow (- \infty,\infty]$ be a bounded below lower semicontinuous function for each $q \in Q$. Suppose that $\{x_{n}\}$ is a sequence in  $E$ such that
$q(x_{n}- x_{n+1}) \leq \phi_{q}(x_{n}) -\phi_{q}(x_{n+1})$ for all $n \in \mathbb{N}_{0}= \mathbb{N} \cup \{0\}$ and   $q \in Q$.
Then $\{x_{n}\}$ converges to a point $v \in E$ and  for each    $q \in Q$, $$q(x_{n}- v) \leq \phi_{q}(x_{n}) -\phi_{q}(v)$$ for all
$n \in \mathbb{N}_{0}$.
\end{lem}
\begin{proof}
Since  $q(x_{n}- x_{n+1}) \leq \phi_{q}(x_{n}) -\phi_{q}(x_{n+1})$, for each $n \in  \mathbb{N}_{0}$ and   $q \in Q$,
 then we have $\{\phi_{q}(x_{n})\}$ is a decreasing sequence for each     $q \in Q$. Moreover, for  $m \in  \mathbb{N}_{0}$,
\begin{align*}
  \sum_{n=0}^{m} q(x_{n}- x_{n+1})  &= q(x_{0}- x_{1}) + q(x_{1}- x_{2}) + \ldots + q(x_{m}- x_{m+1})\\ & \leq  \phi_{q}(x_{0}) - \phi_{q}(x_{m+1})
\\ & \leq  \phi_{q} (x_{0}) - \inf _{n\in \mathbb{N}_{0}}\phi_{q}(x_{n}).
 \end{align*}

Letting $m \rightarrow \infty$, we have
\begin{equation*}
   \sum_{n=0}^{\infty} q(x_{n}- x_{n+1})  < \infty,
\end{equation*}
for each    $q \in Q$, then $\displaystyle \lim_{n} q(x_{n}- x_{n+1})=0$ for each    $q \in Q$. This implies that $\{x_{n}\}$  is a left Cauchy sequence in $E$. Because $E$ is Hausdorff and complete,
there exists a unique  $v \in E$ such that $\displaystyle \lim _{n\rightarrow \infty}x_{n} = v$. Let $m, n  \in  \mathbb{N}_{0}$ with $m > n$. Then for each    $q \in Q$
\begin{align*}
  q(x_{n}- x_{m})& \leq \sum_{i=n}^{m-1} q(x_{i}- x_{i+1}) \\ & \leq \phi_{q}(x_{n})- \phi_{q}(x_{m}).
\end{align*}
Letting $m \rightarrow \infty$, since $\phi_{q}$ is  lower semicontinuous for  each    $q \in Q$ and   from  Theorem 1.4 in \cite{barbu} each $q \in Q$ is   continuous, then for each    $q \in Q$, we conclude that
\begin{equation*}
 q(x_{n}- v)   \leq \phi_{q}(x_{n})- \lim _{m \rightarrow \infty} \phi_{q}(x_{m}) \leq \phi_{q}(x_{n})-   \phi_{q}(v)
\end{equation*}
for all $n  \in  \mathbb{N}_{0}$.
\end{proof}

Now we state an extension of Banach Contraction Principle to   locally
convex spaces and we call  it  Banach $Q$-Contraction Principle.
\begin{thm}\label{contraction} (Banach $Q$-Contraction Principle)
  Suppose that  $Q$ is a   family of   seminorms on a    separated  and   complete        locally
convex space $E$ which determines the topology of $E$
  and $T : E \rightarrow E$ a $Q$-contraction mapping with Lipschitz constant
$k \in (0, 1)$. Then we have the following:\\
\begin{enumerate}
  \item [(a)] There exists a unique fixed point $v \in E$.
   \item [(b)]For arbitrary $x_{0} \in E$, the Picard iteration process defined by
   \begin{equation*}
     x_{n+1} = T({x_{n}}), n \in \mathbb{N}_{0},
   \end{equation*}
converges to $v$.
\item[(c)] $q(x_{n}- v)\leq (1 - k)^{-1} k^{n}q(x_{0}- x_{1})$ for all $n \in \mathbb{N}_{0}$ and  $q \in Q$.
\end{enumerate}
\end{thm}
\begin{proof}
(a) For each $q \in Q$, let   $\phi_{q} : E \rightarrow \mathbb{R^{+}}$ be a functions defined  by $$\displaystyle \phi_{q}(x) = (1 - k)^{-1} q(x- Tx),$$
for each $x \in E$. From  Theorem 1.4 in \cite{barbu}, since  each  $q$ is   continuous, $\phi_{q}$ is also a continuous function.
 From  the fact  that $T$ is a $Q$-contraction mapping, for each $q \in Q$ we have
\begin{equation}\label{solp}
  q(Tx-T^{2}x) \leq kq(x- Tx), x \in E,
\end{equation}
which conclude  that
\begin{align*}
 q(x- Tx) - kq(x- Tx) \leq q(x- Tx) - q(Tx-T^{2}x).
\end{align*}
Hence
\begin{align*}
q(x- Tx) & \leq \frac{1}{1 - k}[q(x- Tx) - q(Tx-T^{2}x)] \\ & \leq \frac{1}{1 - k}[q(x- Tx) - q(Tx-T^{2}x)] ,
\end{align*}
for each $q \in Q$, therefore
\begin{align}\label{bvgfj}
q(x- Tx) \leq  \phi_{q}(x)- \phi_{q}(Tx).
\end{align}
Consider   an arbitrary element $x$   in $X$ and define the sequence ${x_{n}}$ in $E$ by
$x_{n} = T^{n} x, n \in \mathbb{N}$.
From \eqref{bvgfj}, we have
\begin{equation*}
  q(x_{n}- x_{n+1}) \leq \phi_{q}(x_{n}) - \phi_{q}(x_{n+1}), n\in  \mathbb{N},
\end{equation*}
for  each    $q \in Q$  and since $E$ is Hausdorff, it follows from Lemma \ref{qupozm} that there exists an element  $v \in E$ such that
\begin{equation*}
  \lim_{n\rightarrow \infty}x_{n} = v,
\end{equation*}
and
\begin{equation}\label{gjcdl}
  q(x_{n}- v) \leq \phi_{q}(x_{n}), n\in  \mathbb{N}_{0},
\end{equation}
 for  each    $q \in Q$.
Since for example, from page 3 in \cite{barbu}, $T$ is continuous and $x_{n+1} = Tx_{n}$, it follows that $v = Tv$. Suppose that $z$ is
another fixed point for $T$. If for each $q \in Q$ we have $q(v- z)=0$, then from the fact that $E$ is separated, we have $v=z$.  On the other hands,  let  $ 0 < q(v- z)$ for  some     $q \in Q$, then  we have
\begin{equation*}
  0 < q(v- z) = q(Tv-Tz) \leq kq(v- z) <q(v- z),
\end{equation*}
that is a contradiction. Hence $T$ has unique fixed point $v \in E$.

 (b) This assertion follows from part (a).

 (c) From \eqref{solp} we have that $\phi_{q}(x_{n}) \leq k^{n}\phi_{q}(x_{0})$ for  each    $q \in Q$. This implies from \eqref{gjcdl} that
$q(x_{n}- v) \leq k^{n}\phi_{q}(x_{0})$ for  each    $q \in Q$.
\end{proof}

\begin{lem}\label{jcfiol}
Let $E$ be a locally convex space. Then for $x, y \in E$ with $q(x) \neq 0$, the following are equivalent:
\begin{enumerate}
  \item [(a)]  $ q(x) \leq q(x + ty)$ for all $t > 0$ that $q(x + ty) \neq 0$ and $q \in Q$.
  \item [(b)] There exists $j_{q} \in J_{q}x$ such that $\langle y, j_{q} \rangle \geq 0$ for each $q \in Q$.
\end{enumerate}
Proof. (a) $\Rightarrow$ (b).
For $t > 0$, let $f_{t} \in J_{q}(x+ty)$ and define $g_{t} =\frac{f_{t}}{q^{*}({f_{t})}}$. Hence
$q^{*}({g_{t})} = 1$. Since  $g_{t}   \in q^{*}({f_{t})}^{-1} J_{q}(x+ty)$ and $q^{*}({g_{t})} = 1$, we have
\begin{align}\label{bsegk}
 q(x) \leq q(x + ty)& = q^{*}({f_{t})}^{-1} \langle x + ty , f_{t} \rangle \nonumber\\&
 =\langle x + ty , g_{t} \rangle =\langle x   , g_{t} \rangle +t\langle y , g_{t} \rangle \nonumber \\&
\leq q(x) + t\langle y, g_{t}\rangle.
\end{align}
By   Theorem 3.26 in \cite{sob}, the Banach-Alaoglu theorem(which states that the polar $U^{\circ}$ is
weak$^{*}$ly-compact for every   neighborhood of zero in $E$). Putting $U=\{x \in E: q(x)\leq 1\}$, we have $\{g_{t}\} \subset U^{\circ}$ hence, the net $\{g_{t}\}$ has a limit point $g \in E^{*}$ such that
$q^{*}(g) \leq 1$ and from \eqref{bsegk} we have that   $\langle x, g \rangle \geq q(x)$ and $\langle y, g \rangle\geq 0$.
Observe that
\begin{align*}
 q(x)\leq \langle x, g \rangle \leq q(x)q^{*}(g) \leq q(x),
\end{align*}
which gives that
$\langle x, g\rangle = q(x)$ and $q^{*}(g) = 1$.
Set $j_{q} = gq(x)$, then $j_{q} \in J_{q}x$ and $\langle y, j_{q} \rangle \geq 0$.

(b) $\Rightarrow$ (a). Suppose for $x, y  \in X$ with $q(x) \neq 0$ and $q \in Q$, there exists $j_{q} \in  J_{q}x$ such that
$\langle y, j_{q} \rangle \geq 0$. Hence for $t > 0$ that $q(x + ty) \neq 0$,
\begin{equation*}
  q(x)^{2} = \langle x, j_{q} \rangle \leq \langle x, j_{q} \rangle + \langle ty, j_{q} \rangle= \langle x + ty, j_{q} \rangle \leq q(x + ty)q(x),
\end{equation*}
which implies that
$q(x) \leq  q(x + ty)$.
\end{lem}

\begin{thm}\label{sunny}
   Let $C$ be a nonempty convex subset of a separated  locally convex space $X$ and $D$ a nonempty subset of $C$.  Let $J_{q}: E \rightarrow E^{*}$ be single valued   for every $q \in Q$.  Let $ (C-C) \cap \{ x, q(x)=0\}= \{0\}$. If $P$ is a retraction of $C$ onto $D$ such
that for each $q \in Q$,
\begin{equation}\label{hFHFL}
  \langle x - Px , J_{q}(y - Px) \rangle \leq 0, \quad( x \in C, y \in D),
\end{equation}
then $P$ is sunny $Q$-nonexpansive. Conversely, if  $P$ is sunny $Q$-nonexpansive and $ (D-D) \cap \{ x, q(x)=0\}= \{0\}$, then \eqref{hFHFL} holds.
\end{thm}
\begin{proof}
First we show $P$ is sunny: For $x \in C$, put  $x_{t} := Px + t(x - Px)$ for each $t > 0$. Since
$C$ is convex, we conclude that $x_{t} \in C$ for each $t \in (0, 1]$. Hence, from \eqref{hFHFL}, we have
\begin{equation}\label{mvhjllhf}
  \langle x - Px , J_{q}(Px - Px_{t})\rangle \geq 0 \; \text{and}\; \langle x_{t} - Px_{t}, J_{q}(Px_{t} -Px)\rangle \geq 0.
\end{equation}
Because $x_{t}- Px = t(x -Px)$ and $\langle t(x - Px), J_{q}(Px - Px_{t})\rangle  \geq 0$, we have
\begin{equation}\label{glyljl}
  \langle x_{t} - Px , J_{q}(Px - Px_{t})\rangle \geq 0.
\end{equation}
Combining \eqref{mvhjllhf} and \eqref{glyljl}, we have
\begin{align*}
  q(Px - Px_{t})^{2}& = \langle Px - x_{t} + x_{t} - Px_{t} , J_{q}(Px - Px_{t})\rangle
\\ &=  - \langle x_{t}- Px , J_{q}(Px - Px_{t})\rangle + \langle x_{t} - Px_{t}, J_{q}(Px -Px_{t})\rangle
\leq 0.
\end{align*}
Then  $q(Px - Px_{t})=0$. Thus from the fact that $E$ is separated, $Px = Px_{t}$. Therefore, $P$ is sunny.

Now, we show that $P$ is $Q$-nonexpansive: For $x, z \in C$, we have from \eqref{hFHFL} that
$\langle x - Px , J_{q}(Px- Pz) \rangle \geq 0$ and $\langle z - Pz , J_{q}(Pz - Px)\rangle \geq 0$, for each  $q \in Q$.
Hence
$\langle x- z - (Px - Pz) , J_{q}(Px- Pz)\rangle \geq 0$.
Hence,  we conclude  that
$\langle x - z, J_{q}(Px - Pz) \rangle \geq q(Px - Pz)^{2}$ for each  $q \in Q$.
  Therefore, since $ (C-C) \cap \{ x, q(x)=0\}= \{0\}$, it follows that $P$ is $Q$-nonexpansive.

Conversely, suppose that  $P$ is  the sunny $Q$-nonexpansive   retraction and $x \in C$.
Then $Px \in D$ and there exists a point  $z \in D$ such that $Px = z$. Putting  $M :=
\{z + t(x - z) : t \geq 0\}$ we conclude  $M$ is a nonempty convex set. Since $P$ is sunny, i.e., $Pv = z$,  for each  $v \in M$ we have
\begin{align*}
  q(y - z )&= q(Py - Pv)\\&
\leq q(y - v) = q(y - z + t(z - x))
\end{align*}
 for all $y \in D$.
Hence,  from Lemma \ref{jcfiol}, we have
\begin{equation*}
  \langle x - Px , J_{q}(y - Px) \rangle \leq 0, \quad( x \in C, y \in D).
\end{equation*}
\end{proof}

To prove Theorem \ref{lmt}, we need to prove  the following tree corollaries  of Theorem 6.5.3 in \cite{tn}.
\begin{co}\label{waml}
   Suppose that  $Q$ is a   family of   seminorms on a       locally
convex space $E$ which determines the topology of $E$. Let $K$ be a convex subset of    $E$. Let  $x \in E$ and  $x_{0}\in K$ such that $(E-\{x\})\cap \{y \in E: q(y)=0\}=\{0\}$, for each $q \in Q$. Then the following are equivalent, for each $q \in Q$:
\begin{enumerate}
  \item $q(x_{0}-x)=\inf\{ q(x-y): y \in K\}$;
  \item  there exists an $L \in E^{*}$ such that $q^{*}(L)=1$ and
  \begin{equation*}
    \inf\{ L(y-x): y \in K\}= q(x_{0}-x);
  \end{equation*}
  \item   there exists an $L \in E^{*}$ such that $q^{*}(L)=1$ and
  \begin{equation*}
    \inf\{ L(y): y \in K\}= L(x_{0}),
  \end{equation*}
  and
  \begin{equation*}
    L(x_{0}-x)=q(x_{0}-x).
  \end{equation*}
\end{enumerate}
\end{co}
\begin{proof}
 By Theorem 6.5.3 in \cite{tn}, the corollary holds. We just need to show $q^{*}(L)=1$ and $L$ is continuous  in the assertion (1)$\Rightarrow$ (2). We know that, for each $\epsilon > 0$ there exists a point $y_{0} \in K$ such that
 \begin{equation*}
   q(y_{0}-x) \leq   q(x_{0}-x)+\epsilon\leq L(y_{0}-x),
 \end{equation*}
 then by our assumption, we have
 \begin{equation*}
   1\leq L(\frac{y_{0}-x}{q(y_{0}-x)})\leq q^{*}(L)\leq 1,
 \end{equation*}
 hence $q^{*}(L)=1$. To show the continuity of $L$, let $x_{\alpha}$  be a net in $E$ such that $x_{\alpha}\rightarrow z_{0}$, then
 \begin{equation*}
   \langle x_{\alpha}- z_{0}, L\rangle \leq q(x_{\alpha}-z_{0}),
 \end{equation*}
 then  $L$ is continuous.
\end{proof}
\begin{co}\label{ircbem}
Suppose that  $Q$ is a   family of   seminorms on a   real    locally
convex space $E$ which determines the topology of $E$. Let $K$ be a convex subset of    $E$. Let  $x \in E$ and  $x_{0}\in K$ such that $(E-\{x\})\cap \{y \in E: q(y)=0\}=\{0\}$, for each $q \in Q$. Then the following are equivalent:
\begin{enumerate}
  \item $q(x_{0}-x)=\inf\{ q(x-y): y \in K\}$, for each $q \in Q$;
  \item  there exists an $f \in J_{q}(x-x_{0})$ such that
  \begin{equation*}
    \langle x_{0}-y, f\rangle \geq 0,
  \end{equation*}
   for every $y \in K$ for each $q \in Q$.
\end{enumerate}

\end{co}
\begin{proof}
 (1)$\Rightarrow$ (2): Let $q(x_{0}-x)=\inf\{ q(x-y): y \in K\}$, for each $q \in Q$. Then from theorem \ref{waml},  there exists an $L \in E^{*}$ such that $q^{*}(L)=1$ and
  \begin{equation*}
    \inf\{ L(y-x): y \in K\}= q(x_{0}-x).
  \end{equation*}
  Set $f=-q(x_{0}-x)L$. Therefore, since  $q(x_{0}-x)\leq L(x_{0}-x)\leq q(x_{0}-x)$, we have
  \begin{equation*}
    f(x-x_{0})=-q(x_{0}-x)L(x-x_{0})=q^{2}(x_{0}-x)= {q^{*}}^{2}(f).
  \end{equation*}
  Then $f \in J_{q}(x-x_{0})$. Hence, we have, for every $y \in K$,
  \begin{align*}
     f(x_{0}-y)  &=  f(x_{0}-x+x-y)=- q^{2}(x_{0}-x)-q(x_{0}-x)L(x-y)\\ & = - q^{2}(x_{0}-x)+q(x_{0}-x)L(y-x) \\ &  \geq - q^{2}(x_{0}-x)+q^{2}(x_{0}-x)=0.
  \end{align*}
  (2)$\Rightarrow$ (1): If $q(x_{0}-x)=0$ then (1) holds. Hence,  assume that $q(x_{0}-x) \neq 0$. Therefore, we have
  \begin{align*}
      q^{2}(x_{0}-x) &=  \langle x-x_{0} , f \rangle= \langle x-y + y-x_{0} , f \rangle
\\ & \leq  q(x-y)q^{*}(f)+\langle y-x_{0} , f \rangle \leq q(x-y)q^{*}(f)=q(x-y)q(x_{0}-x),
  \end{align*}
  hence we have
  \begin{equation*}
   q(x_{0}-x) \leq q(x-y),
  \end{equation*}
 for all $y \in K$. Then, we have
  $q(x_{0}-x)=\inf\{ q(x-y): y \in K\}$, for each $q \in Q$.
\end{proof}
\begin{co}\label{trsfpl}
Suppose that  $Q$ is a   family of   seminorms on a       locally
convex space $E$ which determines the topology of $E$.  Suppose that $C$ is a nonempty closed convex   subset of $E$. Let $C_{0}\subseteq C$ and $P$ be a sunny  $Q$-nonexpansive retraction of $C$ onto $C_{0}$. Let   $J_{q}$ be single valued duality mapping for each $q \in Q$.
 Then for any $x \in C$ and $y \in C_{0}$,
\begin{equation*}
  \langle x - Px , J_{q}(y - Px) \rangle \leq 0.
\end{equation*}
\end{co}
\begin{proof}
 Suppose  $x \in C$ and $y \in C_{0}$. Set $x_{t}=Px+t(x-Px)$ for each $0\leq t\leq 1$. Then  we have $x_{t} \in C$ and $q(y-Px)=q(Py-Px_{t})\leq q(y-x_{t})$ for each $q \in Q$. By Corollary \ref{ircbem} we have
 \begin{equation*}
     \langle  Px- x_{t} , J_{q}(y - Px) \rangle \leq 0,
 \end{equation*}
 hence,
 \begin{equation*}
     \langle  x-Px , J_{q}(y - Px) \rangle \leq 0.
 \end{equation*}
\end{proof}
\begin{thm}\label{continous}
  Let $E$ be a locally convex space and $J_{q} : E \rightarrow E^{*}$  a single-valued
duality mapping. Then $J_{q}$ is  continuous  from $\tau_{Q}$ to weak$^{*}$ topology.
\end{thm}
\begin{proof}
We show that if $x_{\alpha} \rightarrow x$  in $\tau_{Q}$  then  $J_{q} x_{\alpha} \rightarrow J_{q} x$ in the weak$^{*}$ topology.

 First, assume that $q(x)=0$, then $J_{q} x=0$. We need to show that $J_{q} x_{\alpha} \rightarrow 0$ in the weak$^{*}$ topology.
But,  we have $ q(x_{\alpha}) = q^{*}(J_{q} x_{\alpha})$, and since from Theorem 1.4 in  \cite{barbu} each $q \in Q$ is continuous, we have
$  \displaystyle 0=\lim_{\alpha}q(x_{\alpha}) = \lim_{\alpha} q^{*}(J_{q} x_{\alpha})$, hence,  $\displaystyle\lim_{\alpha} q^{*}(J_{q} x_{\alpha})=0$, therefore, $\displaystyle\lim_{\alpha} \langle y , J_{q} x_{\alpha} \rangle=0$, for every $y \in E$ that $q(y)\leq 1$. Then  for every $y \in E$,  $\displaystyle\lim_{\alpha} \langle y , J_{q} x_{\alpha} \rangle=0$. Indeed, if $q(y)> 1$, then $q(\frac{y}{q(y)})\leq 1$, therefore
$\displaystyle\lim_{\alpha} \langle \frac{y}{q(y)} , J_{q} x_{\alpha} \rangle=0$, then    for every $y \in E$,  $\displaystyle\lim_{\alpha} \langle y , J_{q} x_{\alpha} \rangle=0$. Hence  $J_{q} x_{\alpha} \rightarrow 0$ in the weak$^{*}$ topology.

Second,  assume that $q(x) \neq 0$.
 Set $f_{\alpha}^{q} := J_{q}x_{\alpha}$.
  Then
\begin{align*}
  \langle x_{\alpha}, f_{\alpha}^{q} \rangle = q( x_{\alpha})q^{*}(f_{\alpha}^{q}), \quad q(x_{\alpha}) = q^{*}(f_{\alpha}^{q}).
\end{align*}
Because $(x_{\alpha})$ is convergent in $\tau_{Q}$, $({x_{\alpha}})$ is bounded with respect to $\tau_{Q}$, then we can conclude that  ${f_{\alpha}^{q}}$ is bounded in $E^{*}$.  Indeed, since  $q^{*}({f_{\alpha}^{q}})=q(x_{\alpha})$ and    we know from Definition  1.1 in \cite{conway}  that $E^{*}$ is a l.c.s  that    the seminorms $\{p_{y}: y \in E\}$  which $p_{y}(x^{*})=|\langle y, x^{*}\rangle|$, define   the weak$^{*}$ topology on it, hence,
\begin{equation}\label{oiuvcx}
p_{y}({f_{\alpha}^{q}})=|\langle y , f_{\alpha}^{q} \rangle| \leq q(y)q^{*}({f_{\alpha}^{q}})=q(y)q(x_{\alpha})
\end{equation}
when  $q(y)\neq 0$, and in other words, when   $q(y)= 0$, from the definition of $J_{q}x_{\alpha} $ and $q^{*}({f_{\alpha}^{q}})$ we have that
\begin{equation}\label{npen}
|\langle y , f_{\alpha}^{q} \rangle| \leq  q^{*}({f_{\alpha}^{q}})=q(x_{\alpha})
\end{equation}
 hence, from \eqref{oiuvcx} and \eqref{npen} and from the boundedness concept in locally convex spaces,    page 3 in \cite{barbu}, and since $({x_{\alpha}})$ is bounded then $({f_{\alpha}^{q}})_{\alpha}$  is bounded in $E^{*}$ for each $q \in Q$.  From \eqref{oiuvcx} and \eqref{npen}, we can select an upper bound $M_{q, y} \neq 0$ related to  each $q \in Q$ and $y \in E$ for  $({f_{\alpha}^{q}})_{\alpha}$ in  the weak$^{*}$ topology. Putting $U_{q}=\{z \in E:q(z)< 1 \}$, we have $({\frac{1}{M_{q, y}}f_{\alpha}^{q}})_{\alpha} \subset U_{q}^{\circ}$ for  each $q \in Q$ and $y \in E$   where $U_{q}^{\circ}$ is the polar of $U_{q}$. Then  from Theorem 3.26 (Banach-Alaoglu) in \cite{sob}, there exists a subnet
$(\frac{1}{M_{q, y}}f_{\alpha_{\beta}}^{q})_{\beta}$ of $(\frac{1}{M_{q, y}}f_{\alpha})_{\alpha}$ such that $\frac{1}{M_{q, y}}f_{\alpha_{\beta}}^{q}\rightarrow f \in  U_{q}^{\circ}$  in the weak$^{*}$ topology.

We know that the function  $q^{*}$ on $E^{*}$ is lower semicontinuous in weak$^{*}$ topology. Indeed,  if $g_{\beta} \in \{f \in E^{*}:q^{*}(f)\leq \alpha \}$ such that $g_{\beta} \rightarrow g$ in the weak$^{*}$ topology, then equivalently,   from Definition  1.1 in \cite{conway} and   page 3 in \cite{barbu}, $p_{y}(g_{\beta}-g)\rightarrow 0$  for each   $y \in E$, hence
 $|\langle y ,  g_{\beta}-g\rangle|\rightarrow 0$,  therefore, $ \langle y ,  g_{\beta}\rangle \rightarrow \langle y , g\rangle$, hence $q^{*}(g)\leq \alpha$ i.e, $g \in \{f \in E^{*}:q^{*}(f)\leq \alpha \}$     then from Proposition 2.5.2 in \cite{Ag},  $q^{*}$  is lower semicontinuous in  the weak$^{*}$ topology on $E^{*}$. Therefore, we have
\begin{equation}\label{wsaml}
  q^{*}(f) \leq \liminf _{\beta}q^{*}(\frac{1}{M_{q, y}}f_{\alpha_{\beta}}^{q}) = \frac{1}{M_{q, y}} \liminf _{\beta} q(x_{\alpha_{\beta}}) = \frac{1}{M_{q, y}}q(x).
\end{equation}
Because $\langle x, M_{q, y}f - f_{\alpha_{\beta}}^{q} \rangle \rightarrow 0$ and $\langle x - x_{\alpha_{\beta}}, f_{\alpha_{\beta}}^{q} \rangle\rightarrow 0$, indeed from pages 125 and 126 \cite{conway}, the weak topology   on a l.c.s is coarser than the original    topology, hence,   from the fact that  $x_{\alpha}\rightarrow x$ in the original    topology, therefore  $x_{\alpha}\rightarrow x$ in the weak topology. Now, we have
\begin{align*}
  |\langle x, M_{q, y}f \rangle -q(x_{\alpha_{\beta}})^{2}| &= |\langle x, M_{q, y}f \rangle - \langle x_{\alpha_{\beta}}, f_{\alpha_{\beta}}^{q}
\rangle| \\
&\leq |\langle x, M_{q, y}f - f_{\alpha_{\beta}}^{q}\rangle| + |\langle x - x_{\alpha_{\beta}} , f_{\alpha_{\beta}}^{q}\rangle | \rightarrow 0,
\end{align*}
and since from Theorem 1.4 in  \cite{barbu} each $q \in Q$ is continuous, we have
\begin{equation*}
  \langle x, M_{q, y}f\rangle = q(x)^{2}.
\end{equation*}
Since $q(x) \neq 0$, we have
\begin{equation*}
  q(x)^{2} = \langle x, M_{q, y}f \rangle \leq  q^{*}(M_{q, y} f)q(x).
\end{equation*}
Thus, using \eqref{wsaml}, we have $\langle x, M_{q, y} f\rangle = q(x)^{2}, q(x) = q^{*}(M_{q, y}f)$. Therefore, $M_{q, y}f = J_{q}x$.

\end{proof}
In the next theorem,  we prove an  existence theorem of a sunny $Q$-nonexpansive retract.
\begin{thm}\label{lmt}
 Suppose that  $Q$ is a   family of   seminorms on a   real separated  and   complete   locally
convex space $E$ which determines the topology of $E$.  Let $ (C-C) \cap \{ x, q(x)=0\}= \{0\}$,  for each $q \in Q$.   Suppose that $C$ is a nonempty closed convex and bounded  subset of $E$ such that every sequence
in $C$ has a convergent subsequence.  Let   $T$ be a        $Q$-nonexpansive mapping  such that $  \rm{Fix}(T)\neq \emptyset$.  Let $J_{q}: E \rightarrow E^{*}$ be single valued  for every $q \in Q$.  Then $\rm{ Fix}(T) $ is a
sunny $Q$-nonexpansive retract of $C$ and the sunny $Q$-nonexpansive retraction of $C$ onto $\rm{ Fix}(T) $ is unique.
\end{thm}

\begin{proof}

  By theorem \ref{contraction} as in the proof of step 1 in Theorem \ref{g1} that we will prove, we put
a sequence  $\{z_{n}\}$ in $C$ as follows:
\begin{align}\label{zn1l}
      z_{n}=\frac{1}{n}x+(1- \frac{1}{n})Tz_{n}\quad ( n \in \mathbb{N}),
\end{align}
where,  $x \in C$ is fixed.
 Then  we have
 \begin{align}\label{gab}
  \lim_{n\rightarrow \infty}q(z_{n}-Tz_{n})=0.
 \end{align}
for each  $q \in Q$. Because, from the fact that $C$ is   bounded and $T$ is $Q$-nonexpansive, we have
\begin{align*}
\lim_{n} q(Tz_{n}-z_{n}) &= \lim_{n} q(Tz_{n}-  \frac{1}{n}x-(1- \frac{1}{n})Tz_{n})\\ &= \lim_{n}\frac{1}{n} q(  x - Tz_{n})=0,
\end{align*}
for each $q \in Q$.

Next, we show that the sequence $\{z_{n}\}$   converges    to an element of $\rm{Fix}(T) $. In the other words,  we   show that the    limit set of $\{z_{n}\}$ (denoted by  $\mathfrak{S}\{z_{n}\}  $) is a subset of $\rm{Fix}(T)$.
For each
$z \in  \rm{Fix}(T)$, $ n \in N$ and $q \in Q$, we have
\begin{align*}
 \langle  z_{n}-x ,J_{q}(z_{n}-z)\rangle    \leq 0.
\end{align*}
Indeed, we have for each $z \in \rm{Fix}(T)$,
\begin{align*}
 \langle  z_{n}-x , J_{q}(z_{n}-z)\rangle  =& \langle \frac{1}{n}x+(1- \frac{1}{n})Tz_{n} -x , J_{q}(z_{n}-z)\rangle \\  =& (n-1)\langle  Tz_{n}  - z_{n} , J_{q}(z_{n}-z)\rangle \\ =& (n-1)\langle  Tz_{n}  - Tz , J_{q}(z_{n}-z)\rangle \\&+ (n-1)\langle z  - z_{n} , J_{q}(z_{n}-z)\rangle \\ \leq & (n-1) (q( Tz_{n}  - Tz)q(z_{n}-z)-q(z_{n}-z)^{2})\\  \leq &(n-1) (q(z_{n}-z)^{2}-q(z_{n}-z)^{2})=0.
\end{align*}
Furthermore, we have, for each   $q \in Q$,
$\displaystyle \lim_{n} q(Tz_{n}-z_{n})=0$. Because  from the fact that  $C$ is bounded and $T$ is $Q$-nonexpansive,  we have that $ \{x  - Tz_{n}\}$ is bounded then
\begin{align}\label{hbclak}
\lim_{n} q(Tz_{n}-z_{n}) &= \lim_{n} q(Tz_{n}-  \frac{1}{n}x-(1- \frac{1}{n})Tz_{n})\nonumber\\ &= \lim_{n} \frac{1}{n} q(   x - Tz_{n})=0
\end{align}
for each $q \in Q$.
From our assumption, $ \{z_{n}\} $ has a subsequence converges to a point in $C$.
Let $\{z_{n_{i}}\}$ and $\{z_{n_{j}}\}$ be subsequences of $ \{z_{n}\} $ such that $\{z_{n_{i}}\}$ and $\{z_{n_{j}}\}$ converge
  to $y$ and $z$, respectively.  Therefore,     $y,z \in \rm{Fix}(T)$.  Because from \eqref{hbclak},  for each   $q \in Q$, we have
  \begin{align*}
 q(y-Ty)&\leq  q( y -z_{n_{i}})+ q(z_{n_{i}}-T z_{n_{i}})+ q(T z_{n_{i}} -Ty)\\& \leq 2q( y -z_{n_{i}})+ q(z_{n_{i}}-T z_{n_{i}}).
\end{align*}
Taking limit, since $E$ is separated, we have $y  \in \rm{Fix}(T)$ and similarly $z \in \rm{Fix}(T)$. Further, we have
 \begin{align}\label{la2}
 \langle  y-x , J_{q}(y-z)\rangle= \lim _{i \rightarrow \infty }\langle  z_{n_{i}}-x ,  J_{q}(z_{n_{i}}-z)\rangle   \leq 0.
\end{align}
 Indeed, from the fact that $J_{q}$ is single valued and  since  from Theorem \ref{continous}, $J_{q}$ is  continuous  from $\tau_{Q}$ to weak$^{*}$ topology, we have
 \begin{align*}
  | \langle  z_{n_{i}}-x\,&,\,J_{q}(z_{n_{i}}-z)\rangle  -\langle  y-x\,,\,J_{q}(y-z)\rangle | \nonumber \\=&  | \langle  z_{n_{i}}-x\,,\,J_{q}(z_{n_{i}}-z)\rangle -\langle  y-x\,,\,J_{q}(z_{n_{i}}-z)\rangle \nonumber \\&+ \langle  y-x\,,\,J_{q}(z_{n_{i}}-z)\rangle - \langle y-x\,,\,J_{q}(y-z)\rangle | \nonumber \\ \leq & | \langle  z_{n_{i}}-y\,,\,J_{q}(z_{n_{i}}-z)\rangle| + |\langle  y-x\,,\,J_{q}(z_{n_{i}}-z) -J_{q}(y-z)\rangle | \nonumber \\ \leq & q(  z_{n_{i}}-y ) q((z_{n_{i}}-z)) + |\langle  y-x , J_{q}(z_{n_{i}}-z) -J_{q}(y-z)\rangle | \nonumber \\ \leq & q(  z_{n_{i}}-y ) M_{q} + |\langle  y-x\,,\,J_{q}(z_{n_{i}}-z) -J_{q}(y-z)\rangle |,
 \end{align*}
 where $M_{q}$ is an upper bound for $\{z_{n_{i}}-z\}_{i \in \mathbb{N}}$  for each   $q \in Q$. Hence,  for each   $q \in Q$, we have
  \begin{equation*}
    \langle  y-x , J_{q}(y-z)\rangle = \lim_{i\rightarrow \infty}  \langle  z_{n_{i}}-x , J_{q}(z_{n_{i}}-z)\rangle  \leq 0.
  \end{equation*}
Similarly  $ \langle  z-x , J_{q}(z-y)\rangle  \leq 0$  and therefore $y = z$. Indeed, since  $J_{q}(y-z)=-J_{q}(y-z)$ we have
 \begin{align*}
   q( y-z ) ^{2}&=  \langle  y-z ,J_{q}(y-z)\rangle = \langle  y-x , J_{q}(y-z)\rangle +  \langle  x-z , J_{q}(y-z)\rangle \\ =& \langle  y-x , J_{q}(y-z)\rangle +  \langle  z-x , J_{q}(z-y)\rangle  \leq 0,
 \end{align*}
then  $q( y-z )=0$, for each $q \in Q$, and since $E$ is      separated,    $y = z$.  Thus, $ \{z_{n}\} $   converges   to an element of  $ \rm{Fix}(T)$.

 Therefore, a mapping $P$ of $C$ into itself can be defined by $Px =  \displaystyle\lim_{n}z_{n}$.   Then  we have, for each $ z \in  \rm{Fix}(T) $,
\begin{align}\label{zxyudmi}
\langle  x-Px\,,\,J_{q}(z-Px)\rangle = \lim_{n\rightarrow\infty}\langle z_{n}- x\,,\,J_{q}(z_{n}-z)\rangle   \leq 0.
\end{align}
It follows from   Theorem \ref{sunny} that $P$ is a sunny $Q$-nonexpansive retraction of $C $ onto $\rm{Fix}(T) $.  Suppose  $R$ is another sunny $Q$-nonexpansive retraction of $C$ onto $\rm{Fix}(T)$. Then, from Corollary \ref{trsfpl}, we have, for each $x \in C$
and $z \in \rm{Fix}(T)$,
\begin{align}\label{klbwf}
 \langle x - Rx , J_{q}(z - Rx)\rangle \leq 0.
\end{align}
Putting $z = Rx$ in \eqref{zxyudmi} and $z = Px$ in \eqref{klbwf}, we have
$ \langle  x - Px , J_{q}(Rx -Px) \rangle\leq   0$ and $\langle  x - Rx , J _{q}(Px - Rx) \rangle \leq 0$
and hence $\langle Rx - Px , J _{q}(Rx-Px)\rangle \leq   0$.  Then we have $q^{2}(Rx - Px)\leq0$  for each $q \in Q$ and since $E$ is separated, this implies $Rx = Px$. This completes the proof.
\end{proof}

\begin{prop}\label{pqbnty}
 Suppose that  $Q$ is a   family of   seminorms on a  separated       locally
convex space $E$ which determines the topology of $E$.   Then
\begin{equation*}
  q(x)^{2}- q(y)^{2} \geq 2\langle x - y, j \rangle
\end{equation*}
    for all $x, y \in E$ and $j \in J_{q}y$ such that $q(y) \neq 0$.
\end{prop}
\begin{proof}
   Let $j \in  J_{q}x$, $x \in E$. Then
\begin{align*}
  q(y)^{2} - q(x)^{2} &- 2\langle  y - x , j \rangle \\&=  q(x)^{2} + q(y)^{2}- 2\langle  y   , j \rangle \\&\geq  q(x)^{2} + q(y)^{2}-  2q(x)q(y) \\&\geq  (q(x)  - q(y))^{2}\geq 0.
\end{align*}
\end{proof}
\begin{thm}\label{g1}
 Suppose that  $Q$ is a   family of   seminorms on a real separated   and complete     locally
convex space $E$ which determines the topology of $E$ and    $C$ be a nonempty closed  convex and bounded subset of  $E$ such that  every sequence in $C$ has a convergent subsequence.    Suppose that   $ T$ is a       $Q$-nonexpansive mapping from $C$ into itself such that $  \rm{Fix}(T)\neq \emptyset$.  Assume that $J_{q}$ is  single valued  for each $q \in Q$. Let $ (C-C) \cap \{ x, q(x)=0\}= \{0\}$ for each $q \in Q$.
    Suppose that $f$ is an  $Q$-contraction on $ C$. Let $\epsilon_{n}$ be a sequence in $(0, 1)$ such that $\displaystyle \lim_{n} \epsilon_{n}=0$.
  Then there exists a unique    $ x \in C $ and    sunny  $Q$-nonexpansive retraction $ P $ of $ C $ onto $  \rm{Fix}(T)$  such that the following  net  $\{z_{n}\}$   generated by
\begin{equation}\label{4}
    z_{n}=\epsilon_{n} fz_{n}+(1-\epsilon_{n})Tz_{n}\quad ( n \in  I),
\end{equation}
    converges to $Px$.
\end{thm}
\begin{proof}
Since $f$ is a  $Q$-contraction, there exists $0 \leq \beta< 1 $ such that $q(f(x)-f(y))\leq \beta q(x-y ) $ for each  $x,y \in E$ and $q \in Q$.
We  divide the proof into five   steps.

  Step 1.  The existence of $z_{n}$ which satisfies \eqref{4}.\\Proof. This follows immediately from the fact that for every $ n \in I$, the mapping $N_{n}$ given by
  \begin{align*}
 N_{n}x:=\epsilon_{n} fx+(1-\epsilon_{n}) Tx\qquad(x\in C),
\end{align*}
is a $Q$-contraction.  To see this, put \\$\beta_{n}=(1+\epsilon_{n}(  \beta-1) )$, then  $0 \leq\beta_{n} < 1 \; (n\in\mathbb{N})$. Then we have,
\begin{align*}
q(N_{n}x-N_{n}y)  \leq &\epsilon_{n}   q(fx-fy)+ \left(1-\epsilon_{n}\right)q(Tx-Ty) \\  \leq & \epsilon_{n}  \beta q(x-y)+(1-\epsilon_{n})q(x-y) \\ =&(1+\epsilon_{n}(  \beta-1) )q(x-y)=\beta_{n}q(x-y).
\end{align*}
Therefore, by Theorem \ref{contraction}, there exists a unique point $z_{n}\in C$ such that $N_{n}z_{n}=z_{n}$.

Step 2. $\displaystyle \lim_{n} q(Tz_{n}-z_{n})=0$ for each $q \in Q$. \\
Proof. Since $C$ is bounded, we have that $ \{f z_{n}  - Tz_{n}\}$ is bounded then
\begin{align*}
\lim_{n} q(Tz_{n}-z_{n}) &= \lim_{n} q(Tz_{n}-  \epsilon_{n} fz_{n}-(1-\epsilon_{n})Tz_{n})\\ &= \lim_{n} \epsilon_{n} q(   fz_{n} - Tz_{n})=0
\end{align*}
for each $q \in Q$.

 Step 3.   $\mathfrak{S}\{z_{n}\} \subset \rm{Fix(T) } $, where   $ \mathfrak{S}\{z_{n}\} $ denotes the set of   $\tau_{Q}$-limit points of subsequences of
 $ \{z_{n}\} $.\\
Proof. Let
 $ z\in \mathfrak{S}\{z_{n}\} $,  and let $\lbrace z_{n_{k}}\rbrace $
 be a subsequence of $ \lbrace z_{n}\rbrace $ such
that $ z_{n_{k}}\rightarrow z $. For each $q \in Q$, we have
\begin{align*}
   q( Tz-z) \leq&  q(Tz-Tz_{n_{k}}) + q(Tz_{n_{k}}-z_{n_{k}})+ q( z_{n_{k}}-z) \\ \leq & 2q( z_{n_{k}}-z)+ q(Tz_{n_{k}}-z_{n_{k}}),
\end{align*}
 then by Step 2,
\begin{align*}
   q(Tz-z) \leq 2\lim_{k}  q( z_{n_{k}}-z)+\lim_{j} q(Tz_{n_{k}}-z_{n_{k}})=0,
\end{align*}
for each $q \in Q$, hence $ q(Tz-z)=0$ for each $q \in Q$ and since $E$ is separated  we have   $ z\in \rm{Fix(T) }$.

  Step 4.   There exists a unique sunny $Q$-nonexpansive retraction $ P $  of $ C $ onto $  \rm{Fix}(T)$ and $ x \in C $ such that
\begin{align}\label{gbs}
  K  :=\limsup_{n}\langle x-Px ,J_{q}(z_{n}-Px)) \rangle\leq0.
\end{align}
Proof.  We know  from Theorem \ref{continous} that $J_{q}$ is  continuous  from $\tau_{Q}$ to weak$^{*}$ topology,  then by Theorem \ref{lmt}   there exists a unique sunny $Q$-nonexpansive retraction $ P $ of $ C $ onto
$ \rm{Fix(T)} $.  Theorem  \ref{contraction}  guarantees that   $fP$ has a unique fixed point $x\in C$.  We show that
\begin{equation*}
  K  :=\limsup_{n}\langle x-Px , J_{q}(z_{n}-Px)\rangle\leq0.
\end{equation*}
Note that, from the definition of
 $ K $  and our assumption  that every sequence in $C$ has a   convergent subsequence,  we can select a subsequence  $ \{z_{n_{k}}\}$
 of $ \{z_{n}\}$ with the following properties:\\
(i)  $\displaystyle\lim_{k}\langle x-Px\,,\,J_{q}(z_{n_{k}}-Px)\rangle=K$,\\
(ii) $\{z_{n_{k}}\}$ converges   to a point $z$;\\
by Step 3, we have $ z \in  \rm{Fix}(T) $.  From Theorem \ref{sunny} and since  $J_{q}$ is  continuous,   we have
\begin{align*}
K =& \lim_{j}\langle  x-Px , J_{q}(z_{n_{k}}-Px)\rangle=\langle  x-Px , J_{q}(z-Px)\rangle    \leq 0.
\end{align*}
 Since     $ fPx=x$, we have  $ ( f-I)Px=x-Px $. Now, from Proposition \ref{pqbnty} and our assumption we have, for each $n\in {I}  $,
\begin{align*}
\epsilon_{n}  (   \beta &-1)   q(  z_{n} -Px   ) ^{2} \\ \geq &
\Big [ \epsilon_{n}  \beta   q(  z_{n} -Px )+(1-\epsilon_{n}) q( z_{n}-Px ) \Big]^{2}-q( z_{n}-Px )^{2} \\  \geq &
\Big [ \epsilon_{n}   q(  fz_{n} -f(Px) )+(1-\epsilon_{n}) q(Tz_{n}-Px ) \Big]^{2}-q( z_{n}-Px )^{2} \\  \geq &
2\Big \langle \epsilon_{n}  \Big ( fz_{n}-f(Px)  \Big)\\ &+ (1-\epsilon_{n})(Tz_{n}-Px)-(z_{n}-Px)\, , \,J_{q}(z_{n}-Px) \Big \rangle \\ = & -2\epsilon_{n} \langle  (  f-I)Px\, , \,J_{q}(z_{n}-Px)\rangle  \\ = & -2\epsilon_{n} \langle x-Px\, , \,J_{q}(z_{n}-Px)\rangle ,
\end{align*}
hence,
\begin{align}\label{gbh}
 q( z_{n}-Px )^{2}\leq \frac{2}{1-  \beta }\langle x-Px\, , \,J_{q}(z_{n}-Px)\rangle.
\end{align}
for each $q \in Q$.

Step 5. $ \{z_{n}\}$    converges to $Px$ in $\tau_{Q}$.\\
Proof. Indeed, from \eqref{gbs}, \eqref{gbh} and that $ Px \in  \rm{Fix}(T)$, we conclude
\begin{align*}
    \limsup_{n}q(z_{n}-Px)^{2}\leq &  \frac{2}{1-  \beta}\limsup_{n}\langle x-Px \,,\,J_{q}(z_{n}-Px)\rangle\leq0,
\end{align*}
for each $q \in Q$. That is $z_{n} \rightarrow Px$  in $\tau_{Q}$.
\end{proof}
\section{Numerical example}
\example \label{g22221} Consider  Theorem \ref{g1}. Let $E=\mathbb{R}$ and $q(x)=|x|$ be the only seminorm on $\mathbb{R}$ i.e, $Q=\{|.| \}$.
 Let $C=[0,2]$ and  $T(x)=\frac{1}{2}x+1$ be a $Q$-nonexpansive mapping and $f(x)= \frac{1}{2}x$  be a $Q$-contraction mapping  on $[0,2]$ in  Theorem \ref{g1}.   We know that $P(x)=2$ is a  sunny  $Q$-nonexpansive  retraction from $[0,2]$ onto $\text{Fix}(T)=\{2\}$, then obviously, we have the sequence $\{z_{n}\}_{n \in \mathbb{N}}=\{2-\frac{2}{n}\}_{n \in \mathbb{N}}$ generated by \eqref{4}  converges to $P1=2$.

  \section{Discussion}
In this paper, we establish some  fixed point theorems  by defining the notation of the $q$-duality mappings in locally convex spaces. One can see that Theorem  \ref{g1} is a generalization  of Theorem 3.1 obtained in  \cite{soori2}, our Banach $Q$-Contraction Principle  is a generalization of the well known  Banach Contraction Principle, Theorem \ref{continous}  is one of the most  applicable theorems in this paper that has proved by using Theorem 3.26 (Banach-Alaoglu) in \cite{sob} and is a generalization of Theorem 2.6.10 in \cite{Ag} and  Corollaries \ref{sunny} and \ref{trsfpl} are  the generalizations of very well known   Propositions  2.10.20 and 2.10.21 in \cite{Ag}, in locally convex spaces.

 \section{Conclusion}
In this paper, first, we define the notation of   the $q$-duality mappings in locally convex spaces.     Then we prove   some    theorems and corollaries for examples  Banach contraction principle  for locally convex spaces that we called  Banach Q-Contraction Principle. Using this results, we prove an implicit scheme for locally convex spaces.
\section{ Abbreviations}
Not applicable
\section{Declarations }
\subsection{Availability of data and material}
Please contact author for data requests.
\subsection{Funding}
Not applicable
\section{  Acknowledgements}
 The  author is  grateful to  the University of Lorestan for their support.
 \section{Competing interests}
 The author declares that he has no competing interests.
 \section{Authors' contributions}
 All authors contributed equally to the manuscript, read and approved the final manuscript.

\bibliography{mmnsample}
\bibliographystyle{mmn}

\end{document}